\documentclass{amsart}
\usepackage{amsmath}
\usepackage{amssymb}
\usepackage[dvips]{graphicx}
\usepackage[latin1]{inputenc}
\usepackage{amssymb,latexsym}
\newcommand{\R}{\mathbb R}

\newcommand{\N}{\mathbb N}

\newcommand{\Hyp}{\mathbb H}

\newcommand{\Sp}{\mathbb S}

\newcommand{\ot}{\Omega(t)}

\newcommand{\opt}{\Omega^+(t)}
\newcommand{\ld}{\lambda_2}
\newcommand{\dd}{\textup{d}}

\newtheorem{theorem}{Theorem}[section]

\newtheorem{lemma}[theorem]{Lemma}

\newtheorem{proposition}{Proposition}
\newtheorem{remark}{Remark}

\title[Extremal second Dirichlet eigenvalue ]{Where to place a spherical obstacle so as to maximize the second Dirichlet eigenvalue }

\author[Ahmad El Soufi and Rola Kiwan]{}


\email{elsoufi@univ-tours.fr}
\email{RolaKiwan@uaeu.ac.ae}

\begin{document}
\maketitle

\centerline{\scshape Ahmad El Soufi }
\medskip
{\footnotesize
  \centerline{Laboratoire de Math\'ematiques et Physique Th\'eorique,
UMR CNRS 6083} 
   \centerline{Universit\'e François Rabelais de Tours, Parc de Grandmont, F-37200
Tours France}}

\medskip

\centerline{\scshape Rola Kiwan}
\medskip
{\footnotesize
  \centerline{Laboratoire de Math\'ematiques et Physique Th\'eorique,
UMR CNRS 6083} 
   \centerline{Universit\'e François Rabelais de Tours, Parc de Grandmont, F-37200
Tours France}
\smallskip
\centerline{Current address : Department of Mathematical Sciences, College of Science,}
\centerline{UAE University, P.O.B. 17551, Al-Ain, United Arab Emirates}}
\bigskip

\begin{abstract}
We prove that among all doubly connected domains of $\R^n$ bounded by two spheres of given radii, the second eigenvalue of the Dirichlet Laplacian achieves its maximum when the spheres are concentric (spherical shell). The corresponding result for the first eigenvalue has been established by Hersch \cite{H2} in dimension 2, and by  Harrell, Kr\"oger and Kurata \cite{HKK} and  Kesavan \cite{K} in any dimension.

We also prove that the same result remains valid when the ambient space $\R^n$ is replaced by the standard  sphere $\Sp^n$ or the hyperbolic space $\Hyp^n$.

\end{abstract}

\section {Introduction and statement of results}\label{1}
The \emph{Dirichlet} or \emph{fixed membrane} eigenvalue problem in a bounded domain $\Omega\subset \R^n$, i.e., 
\begin{equation}\label{dirichlet}
\left\{
 \begin{array}{rclll}
\Delta u &  =   & -\lambda\ u &\text{in} &\Omega\\
       u& = & 0                &\text{on} & \partial\Omega,
\end{array}
\right.
\end{equation}
admits a purely discrete spectrum
$$\lambda_1(\Omega)  <  \lambda_2(\Omega)  \le \lambda_3(\Omega) \le 
 \cdots \le \lambda_i(\Omega) \le \cdots\rightarrow \infty ,$$
where each eigenvalue is repeated according to its multiplicity.

Eigenvalue optimization problems date from Lord Rayleigh's ``Theory of Sound'' (1894)  where it was suggested that the disk should minimize
the first eigenvalue $\lambda_1$ among all the domains of given measure. Rayleigh's conjecture has been proved in the 1920's independently by Faber \cite{F} and Krahn \cite{Kr}.
The topic became since a very active research field and several eigenvalue optimization results have been obtained under various constraints. For details and a literature review, we refer to the classical books of P\'olya and Szegö \cite {PS} and Bandle \cite{B}, and the review articles by Payne \cite{P}, Ashbaugh \cite{A1,A2} and Henrot \cite{He}.

The case of multi-connected planar domains, i.e. whose boundary admits more than one component, was first considered by Hersch. Using the method of interior parallels,  he proved in  \cite{H2} the following extremal property of annular membranes:\\ ``\emph{A doubly connected fixed membrane, bounded by two circles of given radii, has maximum $\lambda_1$ when the circles are concentric}''.

Here the  small  disk (i.e. the hole) may represent an obstacle to vibration and the problem answered by Hersch may be understood as a particular case of the following optimal placement problem :   given a domain $D$, we seek the optimal position to place an obstacle $B$ of fixed shape inside D in order to maximize or minimize the eigenvalue $\lambda_k$ of the Dirichlet Laplacian on $\Omega=D\setminus B$.

Hersch's result has been extended to any dimension by Harrell, Kr\"oger and Kurata \cite{HKK} and  Kesavan \cite{K}. These authors also proved that $\lambda_1$ decreases when the center of
the small ball (the hole) moves away from the center of the large ball.  Their proofs are based on a technique of domain reflection. As shown in \cite{HKK}, this method allows extensions of  the result to domains $D$ satisfying an ``interior symmetry property''.  
In a recent paper \cite{EK1}, we investigated a problem of placement under a dihedral symmetry assumption on both the domain $D$ and the obstacle $B$. We proved that extremal configurations for $\lambda_1$ correspond to the cases where the axes of symmetry of $B$ coincide with those of $D$.


The main aim of this paper is to establish a Hersch's type extremal property for spherical shells, but with respect to the \emph{second eigenvalue}. Given two positive numbers $R_1>R_0$ and a point $C\in \R^n$, $|C|<R_1-R_0$,  we denote by $\Omega(C)$ the domain of $\R^n$ obtained by removing the ball of radius $R_0$ centered at $C$ from within the ball of radius $R_1$ centered at the origin. 

\begin{theorem}\label{main}
Among all doubly connected domains of $\R^n$ bounded by two spheres of given radii, the spherical shell (concentric spheres) has the largest second Dirichlet eigenvalue. That is,
$$\lambda_2(\Omega(C))\le\lambda_2(\Omega(O)),$$
where the equality holds if and only if $C=O$. 
\end{theorem}

Notice that the optimization results mentioned above concerning $\lambda_1$ rely on the positivity of the first eigenfunction and the Hadamard variation formula of $\lambda_1$ with respect to domain deformations. These two ingredients are of course no more available  as soon as we deal with a higher order eigenvalue (see \cite{EI1} for an approach to evaluate the first variation of an eigenvalue with non-trivial multiplicity).  However, noticing that the domain $\Omega(C)$ admits  hyperplanes of symmetry, we may consider the spectrum $\{\lambda_i^-(\Omega(C))\}_{i\geq 1}$ corresponding to eigenfunctions which are anti-invariant by the reflection with respect to such a hyperplane of symmetry. We observe that the first anti-invariant eigenvalue $\lambda_1^-(\Omega(C))$ is simple (Lemma \ref{simple}) and show that it decreases as $C$ moves away from the origin (Proposition \ref{l1}). The result then follows from the inequality
$\lambda_2(\Omega(C))\le\lambda_1^-(\Omega(C))$ and the fact that the equality $\lambda_2(\Omega(O))=\lambda_1^-(\Omega(O))$ holds for a spherical shell (Lemma \ref{oo}).

All our arguments can be transposed in a more general setting. Indeed, let $\Sp^n$ and $\mathbb{H}^n$ be the standard sphere and the hyperbolic space, respectively. We consider domains obtained by removing a geodesic ball $B_0$ from a geodesic ball $B_1$ such that $\bar B_0\subset B_1 $, and the eigenvalue problem \ref{dirichlet} associated with the Laplace-Beltrami operator on $ B_1 \setminus \bar B_0$. We obtain the following

\begin{theorem}\label{space forms}
Among all doubly connected domains of $\Sp^n$ (resp. $\mathbb{H}^n$) of the form $ B_1 \setminus \bar B_0$, where $B_0$ and $B_1$ are geodesic balls of fixed radii such that $\bar B_0\subset B_1 $, the second Dirichlet eigenvalue achieves its maximal value uniquely when the balls are concentric. 
\end{theorem}

The corresponding result for the first eigenvalue was obtained by Anisa and Aithal  \cite{AA}.

Lastly,  let us mention the somewhat related results of Shen and Shieh, concerning spherical bands, that is, domains of $\Sp^2$ of the form $B_1 \setminus \bar B_0$, where $B_0$ and $B_1$ are \emph{concentric} geodesic disks.  They show that  among all such bands of fixed area, $\lambda_1$ is maximal when the band is symmetric with respect to an equator \cite{SS}. Shieh \cite{shi} proved that this extremal property of symmetric bands remains true for the second eigenvalue, provided the area is less than $2\pi$.   
 


 \section {Monotonicity of the first anti-invariant eigenvalue and proof of Theorem \ref{main}}\label{2}
 
Let $R_0$ and $R_1$ be two real numbers such that $R_1>R_0>0$.  In all the sequel, we will denote by $B_1$ the open ball in $\R^n$ of radius $R_1$ centered at the origin and, for all $t\in [0,R_1-R_0)$, by $B_0(t)$ the open ball in $\R^n$ of radius $ R_0$ centered at the point $(t,0,\dots, 0)$. We set $\Omega(t):=B_1\setminus \bar B_0(t)$ and denote by $$\lambda_1(t)  <  \lambda_2(t) \le \lambda_3(t)\le 
 \cdots \le \lambda_i(t)\le \cdots$$
the nondecreasing and unbounded sequence of its eigenvalues for the Laplace operator with homogeneous Dirichlet boundary condition. 
 
For symmetry reasons, we only need to prove that, for all $t\in (0,R_1-R_0)$,
$$\lambda_2(t) < \lambda_2(0).$$

The domain $\Omega(t)$ is clearly symmetric with respect to any hyperplane passing through the first coordinate axis.  
Let   $S$ denotes the reflection  with respect to the hyperplane $\{x_n=0\}$. The quadratic form domain $H=W^{1,2}_0(\ot)$ of the Dirichlet Laplace operator splits into the direct sum of two invariant subspaces 
$$H=H^+ \oplus H^-$$
with  $H^{\pm}=\{u\in H \; ; \; u\circ S =\pm u\}$. We denote by $\Delta^{\pm}$ the Laplace operators associated with the same quadratic form (that is the Dirichlet energy) restricted to $H^{\pm}$, so that we have
$$\Delta=\Delta^+ \oplus \Delta^-.$$

We denote by $\{\lambda_i^+(t)\}_{i\geq 1}$ and $\{\lambda_i^-(t)\}_{i\geq 1}$ the spectra of $\Delta^+$ and $\Delta^-$, respectively. The spectrum of $\Delta$ is then equal to the re-ordered union of $\{\lambda_i^+(t)\}_{i\geq 1}$ and $\{\lambda_i^-(t)\}_{i\geq 1}$. Since a first eigenfunction of $\Delta$ does not change sign in $\ot$, one necessarily has $\lambda_1(t)=\lambda_1^+(t)<\lambda_1^-(t)$. Thus, the second  eigenvalue is given by 
\begin{equation}\label{inf}
\ld(t)=\inf \{\lambda_1^-(t), \lambda_2^+(t)\}.
\end{equation}

In the case of a spherical shell (i.e. the case where t=0), one has the following
\begin{lemma}\label{oo}
Let $\mu$ be the first eigenvalue and $f$  the first eigenfunction (unique up to scaling) of the following Sturm-Liouville eigenvalue problem:
	\begin{displaymath}
\left\{\begin{array}{rcll}
f''(r) +\frac{n-1}{r} f'(r) -\frac{n-1}{r^2} f(r) =-\mu f(r)\\
f(R_0)=f(R_1)=0.
\end{array}\right.
\end{displaymath}
Then the set of functions $\{\frac{f(|x|)}{|x|}\ x_1, \cdots , \frac{f(|x|)}{|x|}\ x_n\}$ constitutes a basis for the second eigenspace of the spherical shell $\Omega(0)$.
In particular,  
$$\lambda_2(0)=\lambda^+_2(0)=\lambda^-_1(0)=\mu.$$

\end{lemma}

Notice that $\frac{f(|x|)}{|x|}\ x_n\in H^-$ while, $\forall i\le n-1$, $\frac{f(|x|)}{|x|} \ x_i\in H^+$.

This result should likely be known, at least in dimension 2. For the sake of completeness, we give the following short proof.
\begin{proof}[Proof of Lemma \ref{oo}]
The expression of the Laplace operator with respect to polar coordinates $(r,\sigma)\in (R_0, R_1)\times \Sp^{n-1}$ is 
\begin{equation*}
\Delta = \frac{\partial^2 }{\partial r^2} +\frac {n-1}r \frac {\partial }{\partial r} +\frac 1{r^2} \Delta_{S^{n-1}},
\end{equation*}
where $\Delta_{S^{n-1}}$ is the Laplace-Beltrami operator of the standard $(n-1)$-sphere.
Using separation of
variables, one can see that any
eigenfunction is a linear combination of functions
of the form $ f_k (r)g_k(\sigma)$ 
where, $\forall k\in \N$, $g_k$ is an eigenfunction of $\Delta_{S^{n-1}}$ associated with the eigenvalue $\gamma_k=k(n+k-2)$ of $\Delta_{S^{n-1}}$, and $f_k$ is an eigenfunction of the following Sturm-Liouville eigenvalue problem:
	\begin{displaymath}
 (P_k) \; \; \; \left\{\begin{array}{rcll}
f_k''(r) +\frac{n-1}{r} f_k'(r) -\frac{\gamma_k}{r^2} f_k(r) =-\mu(k) f_k(r)\\
f(R_0)=f(R_1)=0.
\end{array}\right.
\end{displaymath}
We denote by $\mu_1(k)<\mu_2(k)\le\mu_3(k)\le \cdots $ the nondecreasing sequence of eigenvalues of the last problem. The spectrum of $\Delta$ is nothing but their re-ordered union, $\{\mu_l(k)\; ; \; k\ge 0\; ,\; l\ge 1\}$. Recall that a second 
eigenfunction admits exactly two nodal domains (Courant's nodal domain theorem). This condition is fulfilled by an eigenfunction $ f_k (r)g_k(\sigma)$ if and only if, either $k=0$ (that is $g_0$ is constant) and $f_0$ is a second eigenfunction of ($P_0$), or $k=1$ (that is $g_1$ is a linear function) and $f_1$ is a first eigenfunction of ($P_1$). In particular, $$\lambda_2(0)= \min \{\mu_2(0), \mu_1(1)\}.$$
Thus, we need to compare the first eigenvalue $\mu_1(1)$ of ($P_1$) with the second eigenvalue $\mu_2(0)$ of ($P_0$). Let $f_0$ be a second eigenfunction of ($P_0$) and let $r_0\in (R_0,R_1)$ be such that $f_0(r_0)=0$. The derivative $h:=f_0'$ of $f_0$ admits two zeros, $r_1\in (R_0,r_0)$ and  $r_2\in (r_0,R_1)$, and, differentiating ($P_0$), one can check that $h$ satisfies    
  \begin{displaymath}
  \left\{\begin{array}{rcll}
h''(r) +\frac{n-1}{r} h'(r) -\frac{\gamma_1}{r^2} h(r) =-\mu_2(0) h(r)\\
h(r_1)=h(r_2)=0.
\end{array}\right.
\end{displaymath}
Comparing with ($P_1$), the eigenvalue monotonicity principle allows us to deduce that $\mu_2(0) > \mu_1(1)$ and, then, $\lambda_2(0)=  \mu_1(1)$. The corresponding eigenfunctions are of the form $f(|x|) L(\frac{x}{|x|})$, where $f$ is a first eigenfunction of $(P_1)$ and $L$ is a linear function. 
  Among these eigenfunctions, the function  $\frac{f(|x|)}{|x|}\ x_n$ belongs to $H^-$. Hence $\lambda_2(0)\ge \lambda^-_1(0)$. Since $\frac{f(|x|)}{|x|}\ x_1$ belongs to $ H^+$ and is changing sign, we necessarily have $\lambda_2(0)\ge \lambda^+_2(0)$. Using \ref{inf}, we get $\lambda_2(0)=\lambda^+_2(0)=\lambda^-_1(0)$.
\end{proof}
\begin{remark}
In \cite{EI1}, the first author and Ilias introduced the notion of extremal domain for  the $k$-th Dirichlet eigenvalue $\lambda_k$ with respect to volume-preserving domain deformations. They showed that a necessary and sufficient condition for a domain $\Omega $ to be extremal for $\lambda_2$  is that there exists a finite family of second eigenfunctions $\{u_1,\dots,u_m\}$ satisfying $\sum_{i=1}^{m} \left({\partial u_i\over \partial \eta}\right)^2 =1$ on $ \partial \Omega $, 
where ${\eta}$ is the unit normal vector field of $\partial\Omega$ . Using the basis of second eigenfunctions given in Lemma \ref{oo}, we deduce that the spherical shell $\Omega (0)$ is an extremal domain  for $\lambda_2$ with respect to any volume-preserving domain deformation (not only those corresponding to the motion of the inner ball inside the large ball).
\end{remark}

We introduce the domain 
$$\opt :=\ot \cap \{x_n>0\}$$
whose first Dirichlet eigenvalue will be denoted $\lambda_1(\opt)$.

\begin{lemma}\label{simple}
$\forall\ t\in (0, R_1-R_0)$, $\lambda^-_1(t)$ is simple and 
$$\lambda_1^- (t) = \lambda_1(\opt).$$
\end{lemma}
\begin{proof}
If $u\in H^-$ is a first eigenfunction of $\Delta^-$ on $\Omega (t)$, then $u$ vanishes on the hyperplane $\{x_n=0\}$. The restriction of $u$ to $\opt$ is an eigenfunction of the  Dirichlet Laplacian in $\opt$, which implies that
$\lambda_1(\opt)\leq \lambda_1^- (t).$
On the other hand, a first Dirichlet eigenfunction of $\opt$ can be reflected antisymmetrically with respect to the hyperplane $\{x_n=0\}$ to give an eigenfunction of $\Delta^-$ in $\ot$. Hence, $\lambda_1(\opt)\geq \lambda_1^- (t)$ and the result follows immediately. 
 \end{proof}
\begin{proposition}\label{l1}
The function $t\mapsto \lambda^-_1(t)$ is (strictly) decreasing on $(0, R_1-R_0)$. 
\end{proposition}

\begin{proof}


Fix a $t$ in $(0, R_1-R_0)$ and let $u(t)$ be the eigenfunction associated with $\lambda^-_1(t)$, chosen to be positive in $\opt$ and to satisfy  
$$\int_{\opt} u(t)^2 dx =1.$$ 
The function $t\mapsto \lambda^-_1(t)=\lambda_1(\opt)$ is a differentiable function of $t$ (see \cite{GS, Re}) and its derivative is given by the following so-called Hadamard formula (see \cite{EI1, GS, Ha, Sc2}):
\begin{equation}\label{hadamard}
\frac{\dd }{\dd t}\lambda^-_1(t)=\frac{\dd }{\dd t}\lambda_1(\opt)=\int_{\partial\Omega^+(t)}
\left|\frac{\partial u(t)}{\partial {\eta_t}}\right|^2 {\eta_t}\cdot{v} \ d\sigma, 
\end{equation}
where ${\eta_t}$ is the inward unit normal vector field of $\partial\Omega^+(t)$  and 
${v}$ denotes the restriction to $\partial\Omega^+(t)$ of the deformation vector field. In our case, this vector field has the form $\varphi(x)\frac{\partial}{\partial x_1}$, where $ \varphi$ is a smooth function that vanishes on $\partial B_1$ and coincides with $1$ along $\partial B_0(t)$. Now,
$\partial\Omega^+(t)=\left(\Omega(t)\cap\{x_n=0\}\right) \cup\Gamma^+_1\cup \Gamma^+_0(t)$, with 
\begin{center}
$\Gamma^+_1:=\partial B_1\cap \{x_n>0\} \ \; \mbox{and} \ \;\Gamma_0^+(t):=\partial B_0(t)\cap \{x_n>0\}$.  
\end{center} 
Since $v=0$ on $\Gamma^+_1$,  ${\eta_t}\cdot{v}=\frac{\partial}{\partial x_n}\cdot\frac{\partial}{\partial x_1}=0$ on $\Omega(t)\cap\{x_n=0\}$, and ${\eta_t}\cdot{v}=\frac{1}{R_0}(x_1-t)$ on $\Gamma^+_0(t)$, the formula \ref{hadamard} reduces to
\begin{equation}\label{hadamard1}
\frac{\dd }{\dd t}\lambda^-_1(t)= \frac{1}{R_0}\int_{\Gamma^+_0(t)}
\left|\frac{\partial u(t)}{\partial {\eta_t}}\right|^2 (x_1-t) \ d\sigma.
\end{equation}

The hyperplane $Z_t:=\{x_1=t\}$ divides $\opt$ in two parts ; we denote by $\Omega^+_s(t)=\opt\cap\{x_1>t\}$ the smallest one. The reflection of $\Omega^+_s(t)$ with respect to $Z_t$ is a proper subset of $\opt$. 
 We introduce the following function defined in $\Omega^+_s(t)$,
$$w(x)={u}(t)(x)- {u}(t)(x^*),$$ 
where $x^*$ stands for the reflection of $x$ with respect to $Z_t$. Since $u(t)$ vanishes on $\partial \opt$ and is positive inside $\opt$,   $w(x)\le0$ for all $x$ in $\partial \Omega^+_s(t)$ and, moreover, $w(x)<0$ for all $x$ in $\Gamma_1\cap\{x_1>t\}$. 
 Therefore, $w$ satisfies the following problem:
\[
\left\{
\begin{array}{rclll}
\Delta w &  =   & -\lambda_1(\opt)\, w &\text{in} &\Omega^+_s(t)\\
       w & \leq & 0                &\text{on} & \partial\Omega^+_s(t).
\end{array}
\right.
\]    
 The function $w$ must be nonpositive everywhere in $\Omega^+_s(t)$. Otherwise, the subdomain $V=\{x\in \Omega^+_s(t)\; ;\; w(x)>0\}$ would have the same first Dirichlet eigenvalue as $\Omega^+ (t)$, that is $\lambda_1(V)=\lambda_1(\Omega^+ (t))$. But, thanks to the reflection with respect to $Z_t$, $\Omega^+ (t)$  would contain  two disjoint copies of $V$ and, then, $\lambda_2(\Omega^+ (t))\le \lambda_1(V)$, which leads to a contradiction.
 
Therefore, $\Delta w \ge 0$ in $\Omega^+_s(t)$ and the maximal value of $w$ (i.e. zero) is achieved on the boundary. Therefore, $w$ achieves its maximum at every point of $\Gamma^+_{0,s}(t):=\Gamma^+_0(t)\cap\{x_1>t\}$, and, due to the Hopf maximum principle (see \cite[Theorem 7, ch.2]{PW}), the normal derivative of $w$ is negative at any point $x$  of $\Gamma^+_{0,s}(t)$, that is  
$$0\le\frac{\partial u(t)}{\partial {\eta_t}}(x)<\frac{\partial u(t)}{\partial {\eta_t}}(x^*).$$

Coming back to Hadamard's formula \ref{hadamard1}, we get (noticing that $x^*_1-t=-(x_1-t)$)
\begin{eqnarray*}
\frac{\dd }{\dd t}\lambda^-_1(t)&=&\frac 1 R_0\int_{\Gamma_0^+(t)}
\left|\frac{\partial {u}(t)}{\partial {\eta_t}}\right|^2 (x_1-t)\,\dd\sigma\\
&=&\frac 1 R_0 \int_{\Gamma^+_{0,s}(t)} \left(
\left|\frac{\partial {u}(t)}{\partial {\eta_t}}(x)\right|^2 
-\left|\frac{\partial {u}(t)}{\partial {\eta_t}}(x^*)\right|^2
\right)(x_1-t)\, \dd\sigma\\
&<& 0
\end{eqnarray*}
which completes the proof.
\end{proof}

 \begin{proof}[Proof of Theorem \ref{main}]
Applying equation (\ref{inf}), Proposition \ref{l1} and Lemma \ref{oo}, respectively, we get, for all $t\in(0,R_2-R_1)$, 
$$\ld(t)\leq \lambda_1^-(t)<\lambda_1^-(0) = \lambda_2(0).$$
\end{proof}


\section{Domains in the sphere and the hyperbolic space} 

We represent the standard sphere $\Sp^n=\{ x\in\R^{n+1} \ ;\ \sum_{i=0}^{n}x_i^2=1\}$ and the hyperbolic space $\Hyp^n=\{ x\in\R^{n+1}\ ;\ x_0>0\ \text{and} \ x_0^2-\sum_{i=1}^{n}x_i^2=1\}$ 
as hypersurfaces of the Euclidean space $(\R^{n+1}, \sum_{i=0}^{n}dx_i^2)$ and the Minkowski space $(\R^{n+1}, -dx_0^2+\sum_{i=1}^{n}dx_i^2)$, respectively, endowed with the induced Riemannian metrics.
In the sequel, we use the same letter $M$ to denote both the standard sphere and the hyperbolic space.

Let $R_0$ and $R_1$ be two real numbers such that $R_1>R_0>0$, and $R_1<\pi$ in the case of $\Sp^n$.  We denote by $B_1$ the open geodesic ball of radius $R_1$ centered at the point $P:=(1,0,\dots, 0)$ and, for all $t\in [0,R_1-R_0)$, by $B_0(t)$ the open geodesic ball of radius $ R_0$ centered at the point $C(t)=(\cos t,\sin t, 0,\dots, 0)\in \Sp^n$ (resp. $C(t)=(\cosh t,\sinh t, 0,\dots, 0)\in \Hyp^n$) of the geodesic ray defined as intersection with $M$ of the $(x_0,x_1)$-plane. We set $\Omega(t):=B_1\setminus \bar B_0(t)$ and denote by $$\lambda_1(t)  <  \lambda_2(t) \le \lambda_3(t)\le 
 \cdots \le \lambda_i(t)\le \cdots$$
the spectrum of the Laplace-Beltrami operator $\Delta$ with Dirichlet boundary condition on $\Omega(t)$. 
 
Again, for symmetry reasons, we only need to prove that, for all $t\in (0,R_1-R_0)$,
$$\lambda_2(t) < \lambda_2(0).$$
The proof follows the same steps as in the Euclidean case. 
 
The domain $\Omega(t)$ is invariant under the reflection, again denoted by $S$, with respect to the hyperplane $\{x_{n}=0\}$, which is an isometry of $M$.  As before, the spectrum of $\Delta$ is the re-ordered union of two spectra, $\{\lambda_i^+(t)\}_{i\geq 1}$ and $\{\lambda_i^-(t)\}_{i\geq 1}$, corresponding to invariant and anti-invariant eigenfunctions.  

In the case $t=0$, the domain $\Omega (0)$ can be parametrized by
$X:(r,\sigma)\in(R_0,R_1)\times \Sp^{n-1}\mapsto (\cos r, \sin r \ \sigma)\in \Sp^n$ in the spherical case, and $X:(r,\sigma)\in(R_0,R_1)\times \Sp^{n-1}\mapsto (\cosh r, \sinh r\ \sigma)\in \Hyp^n$ in the hyperbolic case. 
\begin{lemma}\label{l2}
Let $\mu$ be the first eigenvalue and $f$ the first eigenfunction (unique up to scaling) of the following Sturm-Liouville eigenvalue problem:
	\begin{displaymath}
\left\{\begin{array}{rcll}
f''(r) +(n-1)\frac{a'(r)}{a(r)} f'(r) -\frac{n-1}{a^2(r)} f(r) =-\mu f(r)\\
f(R_0)=f(R_1)=0,
\end{array}\right.
\end{displaymath}
with $a(r)=\sin r$ in the case of $\Sp^n$ and $a(r)=\sinh r$ in the case of $\Hyp^n$.
Then the second eigenspace of $\Omega(0)$ consists of functions $u$ of the form 
$$u(X(r,\sigma))= f(r) L(\sigma),$$
where $L$ is a linear function on $\Sp^{n-1}\subset\R^n$. In particular,  
$$\lambda_2(0)=\lambda^+_2(0)=\lambda^-_1(0)=\mu.$$
\end{lemma}
For a point $x= X(r,\sigma)\in \Omega(0)$, $r$ represents the distance from $x$ to $P$ and $\sigma$ is the projection to the hyperplane $\{ x_0=0\}$ of $  x /{a(r)}$. Thus,  the $n$ functions $$u_1(x)=\frac{f(r(x))}{a (r(x))}\ x_1, \dots , u_n(x)=\frac{f(r(x))}{a (r(x))}\ x_n$$ 
constitute a basis of the second eigenspace of $\Omega(0)$. These functions are all invariant by $S$ except the last one which is anti-invariant.


\begin{proof}[Proof of Lemma \ref{l2}] The Riemannian metric of $\Omega(0)$ is given in the $(r,\sigma)$-coordinates by
$g=dr^2+ a^2(r) g_{S^{n-1}}$, where $g_{S^{n-1}}$ is the standard metric of $\Sp^{n-1}$. The   
expression of the Laplace-Beltrami operator with respect to these coordinates is 
\begin{equation*}
\Delta = \frac{\partial^2 }{\partial r^2} +(n-1)\frac{a'(r)}{a(r)}\frac {\partial }{\partial r} +\frac 1{a^2(r)} \Delta_{S^{n-1}}.
\end{equation*}
Separating the variables and using exactly the same arguments as in the proof of Lemma \ref{oo}, we get the result.
\end{proof}
We introduce 
$\opt=\Omega(t)\cap\{x_n>0\}$ and $\lambda_1(\opt)$ as in Section 2, and check that we still have 
 the simplicity of $\lambda^-_1(t)$ with
$\lambda_1^- (t) = \lambda_1(\opt).$
The proof of Theorem \ref{space forms} will be complete after the following 
\begin{proposition}\label{l3}
The function $t\mapsto \lambda^-_1(t)$ is (strictly) decreasing on $(0, R_1-R_0)$. 
\end{proposition}

\begin{proof}
Hadamard's variation formula remains valid for domains in a general Riemannian manifold (see \cite{EI1}) and gives:
\begin{equation}\label{elsoufi-ilias}
\frac{\dd }{\dd t}\lambda^-_1(t)=\frac{\dd }{\dd t}\lambda_1(\opt)=\int_{\partial\Omega^+(t)}
\left|\frac{\partial u(t)}{\partial {\eta_t}}\right|^2 {\eta_t}\cdot{v} \ d\sigma, 
\end{equation}
where $u(t)$, ${\eta_t}$ and ${v}$ are the first eigenfunction (satisfying $\int_{\opt} u(t)^2  =1$ and  $u(t)>0$ in $\opt$),  the inward unit normal vector field and the deformation vector field on $ \partial\Omega^+(t)$, respectively.\\

\emph{Case of $\Sp^n$}: Let $V(x)=(-x_1, x_0, 0, \dots,0)$ be the Killing vector field of $\Sp^n$ generating rotations in the $(x_0, x_1)$-plane.  The motion of  $B_0(t)$ inside $B_1$ along the geodesic ray $C(t)=(\cos t, \sin t, 0, \dots, 0)$ is generated by a vector field of the form   
$v=\varphi(x)V$, where $ \varphi$ is a smooth function that vanishes on $\partial B_1$ and coincides with $1$ along $\partial B_0(t)$. 
Notice that $v$ is tangent to the geodesic ray $C(t)$. On the other hand,  the unit normal vector $\eta_t(x)$ to $\partial B_0(t)$ at $x$ is nothing but the normalized orthogonal projection of the vector $- C(t)$ to the tangent space $ T_x\Sp^n$, that is,  
$$ \eta_t(x)= -\frac {C(t)-(x\cdot C(t))x}{\sqrt{1-(x\cdot C(t))^2}}.$$
Thus, $\forall x\in  \Gamma^+_0(t):=\partial B_0(t)\cap\{x_n>0\}$,
\begin{eqnarray*}
{\eta_t}\cdot{v}(x)=-\frac {C(t)\cdot V(x)}{\sqrt{1-(x\cdot C(t))^2}}&=& \frac {V(C(t))\cdot x}{\sqrt{1-(x\cdot C(t))^2}}\\
&=&
\frac {x_1 \cos t - x_0 \sin t} {\sqrt{1-(x_0 \cos t + x_1 \sin t)^2} }, 
\end{eqnarray*}
and ${\eta_t}\cdot{v}$ vanishes at any other point of the boundary of $\opt$. Therefore, the formula (\ref{elsoufi-ilias}) reduces to
\begin{equation}\label{elsoufi-ilias1}
\frac{\dd }{\dd t}\lambda^-_1(t)= \int_{\Gamma^+_0(t)}
\left|\frac{\partial u(t)}{\partial {\eta_t}}\right|^2  {\eta_t}\cdot{v}(x)\ d\sigma.
\end{equation}
Consider the hyperplane $Z_t:=\{x\cdot V(t) =  0\}$, with $V(t):=V(C(t))$, and let $\Omega^+_{s}(t)=\opt\cap\{x\cdot V(t)>0\}$ and $\Gamma^+_{0,s}(t)=\Gamma^+_{0}(t)\cap\{x\cdot V(t)>0\}$. The reflection $x^*$ of a point $x$ with respect to $Z_t$ is given by $x^*=x-2 (x\cdot V(t)) \, V(t)$. One can easily check that the image of $\Omega^+_{s}(t)$ by this reflection is a proper subset of $\opt$ and that, $\forall x\in\Gamma^+_0(t)$,  
$${\eta_t}\cdot{v}(x^*)=-{\eta_t}\cdot{v}(x)=\frac {x\cdot V(t)}{\sqrt{1-(x\cdot C(t))^2}}.$$
Thus,
\begin{eqnarray*}
\frac{\dd }{\dd t}\lambda^-_1(t)= \int_{\Gamma^+_{0,s}(t)} \left(
\left|\frac{\partial {u}(t)}{\partial {\eta_t}}(x)\right|^2 
-\left|\frac{\partial {u}(t)}{\partial {\eta_t}}(x^*)\right|^2
\right)\frac {x\cdot V(t)}{\sqrt{1-(x\cdot C(t))^2}} \dd\sigma.\\
\end{eqnarray*}
The same argument used in the proof of Proposition \ref{l1} enables us to show that, at any point $ x\in \Gamma^+_{0,s}(t)$,
$$0\le\frac{\partial u(t)}{\partial {\eta_t}}(x)<\frac{\partial u(t)}{\partial {\eta_t}}(x^*),$$
and, then, $\frac{\dd }{\dd t}\lambda^-_1(t)<0$.\\

\emph{Case of $\Hyp^n$}: The proof is the same as for $\Sp^n$. All arguments and formulas above remain true in the hyperbolic setting with $V(x)=(x_1, x_0, 0,\dots,0)$, $C(t)=(\cosh t, \sinh t, 0,\cdots,0)$, and provided the Euclidean inner product is replaced by the bilinear form $x.y=-x_0 y_0+\sum_{i=1}^n x_i y_i$. 

\end{proof}

\section*{Acknowledgements} 

The authors wish to
thank Evans Harrell and Bernard Helffer for helpful discussions.


\end{document}